\documentclass[12pt]{amsart}

\usepackage{tgpagella}
\usepackage{euler}
\usepackage[T1]{fontenc}
\usepackage{amsmath, amssymb, mathabx}
\usepackage[hidelinks]{hyperref}
\usepackage[english]{babel}
\usepackage{mathrsfs}
\usepackage{eucal}
\usepackage[all]{xy}
\usepackage{tikz}

\usepackage{todonotes}

\newtheorem{thm}{Theorem}[section]
\newtheorem*{thm*}{Theorem}
\newtheorem{lem}[thm]{Lemma}
\newtheorem{fact}[thm]{Fact}

\newtheorem{prop}[thm]{Proposition}
\newtheorem*{prop*}{Proposition}

\newtheorem{cor}[thm]{Corollary}
\newtheorem*{cor*}{Corollary}

\theoremstyle{definition}
\newtheorem{defn}[thm]{Definition}
\newtheorem*{defn*}{Definition}

\newtheorem{remarks}[thm]{Remarks}

\newtheorem*{question*}{Question}
\newtheorem*{Pquestion*}{Popa's question}

\newtheorem*{conv*}{Convention}

\newcommand{\norm}[1]{{\left\lVert #1\right\rVert}}

\def\bb{\mathbb}

\def\bb{\mathbb}

\def\cal{\mathcal}

\def\u{\mathsf 1}

\newcommand{\cstar}{$\mathrm{C}^*$}

\makeatletter

\def\dotminussym#1#2{%
  \setbox0=\hbox{$\m@th#1-$}%
  \kern.5\wd0%
  \hbox to 0pt{\hss\hbox{$\m@th#1-$}\hss}%
  \raise.6\ht0\hbox to 0pt{\hss$\m@th#1.$\hss}%
  \kern.5\wd0}

\DeclareMathOperator{\id}{id}

\DeclareMathOperator{\diag}{diag}

\def \u{\mathcal U}

\newcommand{\mb}{\mathbb}

\def\O{\mathcal{O}}

\def\st{\operatorname{st}}

\allowdisplaybreaks
\makeatletter
\def\l@subsection{\@tocline{2}{0pt}{2.5pc}{5pc}{}}
\def\l@subsubsection{\@tocline{2}{0pt}{5pc}{7.5pc}{}}
\makeatother

\begin{document}


\title{Computably strongly self-absorbing \cstar-algebras}

\author{Isaac Goldbring}
\address{Department of Mathematics\\University of California, Irvine, 340 Rowland Hall (Bldg.\# 400),
Irvine, CA 92697-3875}
\email{isaac@math.uci.edu}
\urladdr{http://www.math.uci.edu/~isaac}
\thanks{Goldbring was partially supported by NSF grant DMS-2054477.}

\begin{abstract}
We introduce the notion of a computably strongly self-absorbing \cstar-algebra and show that the following \cstar-algebras are computably strongly self-absorbing:  the Cuntz algebras $\O_2$ and $\O_\infty$, the UHF algebra $M_{\frak n}(\bb C)$ and the tensor product $M_{\frak n}(\bb C)\otimes \cal O_\infty$, where $\frak n$ is a supernatural number of infinite type with computably enumerable support, and the Jiang-Su algebra $\cal Z$.  In connection with the last example, we show that $\cal Z$ has a computable presentation.  The results above are a special instance of a computable version of the standard approximate intertwining argument due to Elliott.
\end{abstract}

\maketitle

\section{Introduction}

A separable unital \cstar-algebra $A$ is called \textbf{self-absoring} if $A\cong A\otimes A$ \footnote{Unless otherwise stated, $\otimes$ denotes the minimal tensor product of \cstar-algebras.} while it is called \textbf{strongly self-absorbing} if there is an isomorphism $A\to A\otimes A$ that is approximately unitarily equivalent to the inclusion $\id_A\otimes 1_A:A\hookrightarrow A\otimes A$.\footnote{Two $*$-homomorphisms $\varphi,\psi:A\to B$ between separable unital \cstar-algebras $A$ and $B$ are approximately unitarily equivalent if there is a sequence $(v_n)_{n\in \bb N}$ of unitaries such that, for all $a\in A$, $\lim_{n\to \infty}\|v_n\varphi(a)v_n^*-\psi(a)\|=0$.}\footnote{One typically excludes $\bb C$ from the class of strongly self-absorbing \cstar-algebras.}  The class of strongly self-absorbing \cstar-algebras has been intensely studied and, modulo a positive resolution to the \emph{UCT conjecture}, a complete list of the strongly self-absorbing \cstar-algebras has been obtained:  the Cuntz algebras $\cal O_2$ and $\cal O_\infty$, the UHF algebras $M_{\frak n}(\bb C)$ of infinite type, tensor products $M_{\frak n}(\bb C)\otimes \O_\infty$ (with the UHF algebra again of infinite type), and the Jiang-Su algebra $\cal Z$.

In this paper, we introduce an effective version of strongly self-absorbing \cstar-algebras that we call \textbf{computably strongly self-absorbing}, meaning that there is a \emph{computable} isomorphism between $A$ and $A\otimes A$ that is \emph{computably} approximately unitarily equivalent to $\id_A\otimes 1_A$.  In order to state that there is a computable isomorphism between $A$ and $A\otimes A$, one first has to equip $A$ with a \textbf{presentation}, which is simply a countable  sequence from $A$ that generates $A$ as a \cstar-algebra; this presentation naturally induces a tensor product presentation on $A\otimes A$.  To say that an isomorphism $\varphi:A\to A\otimes A$ is computable means that there is an algorithm such that, upon input some $*$-polynomial $p$ in the generators of $A$ and some rational tolerance $\epsilon$, returns some $*$-polynomial $q$ in the generators of $A\otimes A$ such that $\|\varphi(p)-q\|<\epsilon$.  To say that $\varphi$ is computably approximately unitarily equivalent to $\id_A\otimes 1_A$ means that there is a computable sequence of unitaries $(u_n)_{n\in \bb N}$ from $A\otimes A$ (what this means exactly is explained in the next section) such that, upon input a $*$-polynomial $p$ in the generators for $A$ and rational $\epsilon>0$, returns $n$ such that $\|u_n\varphi(a)u_n^*-a\otimes 1\|<\epsilon$.

The main result of this paper is that the standard presentations of all of the strongly self-absorbing \cstar-algebras satisfying the UCT are computably strongly self-absorbing, except that in the case of the algebras involving UHF algebras, our proofs require the additional assumption that the UHF algebra have c.e.\footnote{Here, c.e. stands for \textbf{computably enumerable}; a set $X$ of natural numbers is c.e. if there is an algorithm which, upon input $n$, halts and returns ``yes'' if $n\in X$, while if $n\notin X$, the algorithm either does not halt or halts and returns ''no.''  C.e. has replaced the older terminology of \textbf{r.e.}, which stands for \textbf{recursively enumerable}.} \textbf{support}, where by the support of a UHF algebra we mean the set of primes for which, in the supernatural number associated to the algebra, the exponent is nonzero.

The proof that the algebras mentioned above are strongly self-absorbing uses an approximate intertwining argument originally due to Elliott; see also \cite[Proposition 2.5.3]{rordam}.  We prove a computable version of this result in Section 3 below and use it to deduce the aforementioned results about computably strongly self-absorbing \cstar-algebras.

The results in this paper were originally inspired by the following quote of Blackadar \cite{blackadar}:  ``...in fact, it is in principle essentially impossible to give an explicit isomorphism of $\O_2\otimes \O_2$ and $\O_2$ by the results of \cite{AC}.''  In private communication, Blackadar expanded on this comment by indicating by ``explicit'' he meant ``algebraic'', given that the main result of \cite{AC} states that the $*$-subalgebra $L_2$ (known as the \textbf{Leavitt path algebra}) of $\O_2$ generated by the pair of complementary isometries is such that $L_2\not\cong L_2\otimes L_2$.  The fact that $\cal O_2$ is computably strongly self-absorbing is a contrast to Blackadar's sentiment provided that we change the interpretation of ``explicit'' from ``algebraic'' to ``computable.''

Section 2 contains the necessary background material on computable presentations of \cstar-algebras while the computably approximate intertwining result is proven in Section 3.  Section 4 contains the applications to computably strongly self-absorbing \cstar-algebras while Section 5 provides a proof that the Jiang-Su algebra has a ``standard'' presentation that is computable.

We would like to thank Bruce Blackadar, Bradd Hart, Timothy McNicholl, and Alessandro Vignati for helpful comments concerning this work.

\section{Background on presentations of \cstar-algebras}

Throughout this paper, to simplify matters, we restrict ourselves to unital \cstar-algebras and unit-preserving $*$-homomorphisms.

\subsection{Presentations of \cstar-algebras}

Let $A$ be a separable \cstar-algebra.  A \textbf{presentation} of $A$ is a pair $A^\dagger:=(A,(a_n)_{n\in \mathbb N})$, where $\{a_n \ : \ n\in \mathbb N\}$ is a subset of $A$ that generates $A$ (as a \cstar-algebra).  
Elements of the sequence $(a_n)_{n\in \mathbb N}$ are referred to as \textbf{special points} of the presentation while elements of the form $p(a_{i_1},\ldots,a_{i_k})$ for $p$ a $*$-polynomial with coefficients from $\bb Q(i)$ (a \textbf{rational polynomial}) are referred to as \textbf{rational points} of the presentation.  By fixing an effective bijection between the set of rational polynomials and $\bb N$, we can fix an effective enumeration of the rational polynomials and thus from any presentation of a \cstar-algebra we obtain an effective enumeration of the rational points of the presentation.  If $x$ is the $n^{\text{th}}$ rational point of $A^\dagger$, then we call $n$ an \textbf{code} for $x$.

We say that $A^\dagger$ is a \textbf{computable presentation} of $A$ if there is an algorithm such that, upon input a rational point $p$ of $A^\dagger$ and $k\in \mathbb N$, returns a rational number $q$ such that $|\|p\|-q|<2^{-k}$.  A weaker notion is that of a \textbf{left-c.e.} (resp. \textbf{right}-c.e.) presentation, which means that there is an algorithm which, upon input a rational point $p$ of $A^\dagger$, enumerates a sequence of lower bounds (resp. upper bounds) which converges to $\|p\|$.  Note that a presentation is computable if and only if it is both left-c.e. and right-c.e.  If $A^\dagger$ is a computable, left-c.e., or right-c.e. presentation, then by a \textbf{code} for $A^\dagger$ we mean a natural number which codes the finite sequence of strings that describes the algorithm.  

An element $x\in A$ is a \textbf{computable point} of the presentation $A^\dagger$ if there is an algorithm which, upon input $k\in \mathbb N$, returns a rational point $q$ of $A^\dagger$ such that $\|x-q\|<2^{-k}$.  Once again, one can speak of the code of a computable point of a presentation.

If $A^\dagger$ and $B^\#$ are presentations of \cstar-algebras $A$ and $B$ respectively, a function $\varphi : A^\dagger \to B^\#$ is \textbf{computable} if $\varphi$ is a function from $A$ to $B$ for which there is is an algorithim such that, upon input a rational point $p$ of $A^\dagger$ and $k \in \mb{N}$, returns a rational point $q$ of $B^\#$ such that $\norm{\varphi(p) - q} < 2^{-k}$; in other words, $\varphi$ is a computable map if the $\varphi$-images of rational points of $A^\dagger$ are computable points of $B^\#$, uniformly in the code for the rational point, meaning that the code for the computable point $\varphi(p)$ can be computed from the code for $p$.  Once again, one may speak of the code of a computable map as the code of such an algorithm.  An isomorphism $\varphi:A\to B$ between \cstar-algebras is a \textbf{computable isomorphism} from $A^\dagger$ to $B^\#$ if it is a computable map with computable inverse.  (The computability of the inverse is automatic if $B^\#$ is computable.) 

By a \textbf{computable sequence} in $A^\dagger$ we mean a sequence $(x_n)_{n\in \bb N}$ from $A$ consisting of computable points of $A^\dagger$ for which the function sending $n$ to the code for $x_n$ is computable.

\subsection{Universal presentations of \cstar-algebras}

Operator algebraists might be familiar with a different notion of a presentation of a \cstar-algebra.  In this subsection, we clarify the relationship between these two notions of presentation.

Let $\cal G$ be a set of noncommuting indeterminates, which we call \textbf{generators}.  By a set of \textbf{relations} for $\cal G$ we mean a set of relations of the form $\|p(x_1,\ldots,x_n)\|\leq a$, where $p$ is a $*$-polynomial in $n$ noncommuting variables with no constant term, $x_1,\ldots,x_n$ are elements of $\cal G$, and $a$ is a nonnegative real number.  We also require that, for every generator $x\in \cal G$, there is a relation of the form $\|x\|\leq M$ in $\cal R$.  A \textbf{representation} of $(\cal G,\cal R)$ is a function $j:\cal G\to A$, where $A$ is a \cstar-algebra, such that $\|p(j(x_1),\ldots,j(x_n))\|\leq a$ for every relation $\|p(x_1,\ldots,x_n)\|\leq a$ in $\cal R$.

The \textbf{universal \cstar-algebra} of $(\cal G,\cal R)$ is a \cstar-algebra $A$ along with a representation $\iota:\cal G\to A$ of $(\cal G,\cal R)$ such that, for all other representations $j:\cal G\to B$ of $(\cal G,\cal R)$, there is a unique *-homomorphism $\varphi:A\to B$ such that $\varphi(\iota(x))=j(x)$ for all $x\in \cal G$.  If the universal \cstar-algebra of $(\cal G,\cal R)$ exists, then it is unique up isomorphism and will be denoted by $C^*\langle \cal G|\cal R\rangle$.  Note that $C^*\langle \cal G|\cal R\rangle$ is generated by the image of the generators.  If $\cal G$ is a sequence $\bar x$, then we may write $C^*\langle \bar x | \cal R\rangle$ instead of $C^*\langle \cal{G} | \cal R\rangle$.  Given that we remain in the context of unital \cstar-algebras throughout this paper, we implicitly assume that we have a distinguished generator for the unit and include relations stating that it is a self-adjoint idempotent which acts as a multiplicative identity.

If the \cstar-algebra $A$ is isomorphic to a universal \cstar-algebra $(\cal G,\cal R)$, we refer to $(\cal G,\cal R)$ as a \textbf{generator-relation presentation} of $A$; note that a given \cstar-algebra might admit many generator-relation presentations.  Given a generator-relation presentation $C^*\langle \bar x|\cal R\rangle$ of $A$, we define the corresponding \textbf{universal presentation} of $A$ to be the presentation of $A$ (in the sense of the previous subsection) with $\bar{x}$ as the sequence of special points.  Since we always assume that in any generator-presentation we have an indeterminate for the identity element, it follows that the identity is a special point of any universal presentation of a \cstar-algebra.

Some \cstar-algebras admit ``canonical'' generator-relation presentations.  For example, the Cuntz algebra $\cal O_2$ is most commonly defined as the universal (unital) \cstar-algebra generated by two contractions $s_1$ and $s_2$ subject to the following relations:  $s_1^*s_1=s_2^*s_2=1$ and $s_1s_1^*+s_2s_2^*=1$.  When there is no possible confusion, we call a presentation $A^\dagger$ of $A$ the \textbf{standard presentation} of $A$ if it is the universal presentation corresponding to a canonical generator-relation presentation of $A$; we denote the standard presentation of $A$ by $A^{\st}$.  

A relation $\|p(x_1,\ldots,x_n)\|\leq a$ is called \textbf{rational} if $p$ is a rational polynomial and $a$ is a nonnegative dyadic rational.  A presentation $A^\dagger$ of a \cstar-algebra $A$ is called \textbf{c.e.} if for some c.e. set of rational relations $\cal R$ it is the universal presentation corresponding to $C^*\langle \bar x|\cal R\rangle$.  When $\bar x$ and $\cal R$ are both finite, we say that $A^\dagger$ is \textbf{finitely c.e.}  For example, the standard presentation $\cal O_2^{\st}$ of $\cal O_2$ is finitely c.e.

We will need the following facts, due to Fox \cite[Theorems 3.3 and 3.14]{fox}:

\begin{fact}\label{cefact}
The notions ``c.e. presentation'' and ``right-c.e. presentation'' coincide.  Moreover, from a code for the c.e. set $\cal R$, one can compute a code for the algorithm witnessing that the universal presentation corresponding to $C^*\langle \cal G,\cal R\rangle$ is right-c.e.
\end{fact}

\begin{fact}\label{alec}
If $A$ is a simple \cstar-algebra, then any c.e. presentation $A^\dagger$ of $A$ is computable.
\end{fact}

\subsection{Tensor product and inductive limit presentations}

Consider two presentations $A^\dagger=(A,(a_n)_{n\in \bb N})$ and $B^\#=(B,(b_n)_{n\in \bb N})$ of \cstar-algebras $A$ and $B$.  Given any (\cstar-)tensor product $A\otimes_\alpha B$ of $A$ and $B$, we can consider the \textbf{tensor product presentation} $(A\otimes_\alpha B)^{\dagger\otimes \#}$ of $A\otimes_\alpha B$, given by declaring the $n^{\text{th}}$ special point to be the elementary tensor $a_m\otimes b_p$, where $\langle\cdot,\cdot\rangle$ is a computable pairing of $\bb N^2$ with $\bb N$ and $\langle m,p\rangle=n$. 

Call a presentation $A^\dagger=(A,(a_n)_{n\in \bb N})$ \textbf{bounded} if there is a computable function such that, upon input $n$, returns an upper bound on $\|a_n\|$.  Note that any right-c.e. presentation is bounded.

\begin{lem}\label{computablembedding}
If $A^\dagger$ is a bounded presentation of $A$ and $1$ is a computable point of $B^\#$, then $\id_A\otimes 1_B:A^\dagger\to (A\otimes_\alpha B)^{\dagger\otimes \#}$ is a computable map.
\end{lem}

\begin{proof}
Given a rational point $a$ of $A^\dagger$ and $k\in \bb N$, effectively find an upper bound $M$ on $\|a\|$ (which is possible since the presentation is bounded) and find a rational point $b$ of $B^\dagger$ such that $\|b-1\|<2^{-k}M$; it follows that $\|a\otimes 1-a\otimes b\|<2^{-k}$.  It remains to note that $a\otimes b$ is a rational point of $(A\otimes_\alpha B)^{\dagger\otimes \#}$ whose code can be computed from codes for $a$ and $b$. 
\end{proof}

Suppose now that $A^\dagger$ and $B^\#$ are universal presentations of $A$ and $B$.  Then there is a corresponding universal presentation $(A\otimes_{\max} B)^{u(\dagger\otimes \#)}$ of the maximal tensor product $A\otimes_{\max} B$, whose generators are the union of the generators of $A$ and $B$ (enumerated, as above, via some computable pairing between $\bb N^2$ and $\bb N$) and whose relations are the relations defining $A$ and $B$ individually as well as relations stating that the generators and adjoints of generators of $A$ commute with the generators and adjoints of generators of $B$.  (Note that the implicit relations  defining the identities of $A$ and $B$ individually should be replaced by a single set of relations defining the identity.)  As usual, a generator $x$ of $A$ gets identified with the element $x\otimes 1$ of $A\otimes_{\max} B$ and similarly for generators of $B$.  Although the tensor product presentation $(A\otimes_{\max} B)^{\dagger\otimes \#}$ and universal presentation $(A\otimes_{\max}B)^{u(\dagger\otimes \#)}$ are not literally the same, we nevertheless have:

\begin{lem}\label{doesntmatter}
If $A^\dagger$ and $B^\#$ are universal presentations of $A$ and $B$ respectively, then $\id_{A\otimes_{\max}B}:(A\otimes_{\max} B)^{\dagger\otimes \#}\to (A\otimes_{\max}B)^{u(\dagger\otimes \#)}$ is a computable isomorphism.
\end{lem}

\begin{proof}
    We first note that the identity map is computable.  To see this, note that any special point of $(A\otimes_{\max} B)^{\dagger\otimes \#}$ is either a special point of $(A\otimes_{\max} B)^{u(\dagger\otimes \#)}$ (such as $x\otimes 1$) or a product of two special points of $(A\otimes_{\max} B)^{u(\dagger\otimes \#)}$ (such as $x\otimes y$).  From this observation, it follows easily that the map is computable.  To see that the inverse map is computable, it suffices to note that every special point of $(A\otimes_{\max} B)^{u(\dagger\otimes \#)}$ is a special point of $(A\otimes_{\max} B)^{\dagger\otimes \#}$. 
\end{proof}

The following fact is immediate from the definitions:

\begin{lem}\label{cetensor}
Suppose that $A^\dagger$ and $B^\#$ are c.e. presentations of $A$ and $B$ respectively.  Then $(A\otimes_{\max} B)^{u(\dagger\otimes \#)}$ is a c.e. presentation of $A\otimes_{\max} B$.
\end{lem}

Throughout this paper, we equip $M_n(\bb C)$ with its standard presentation $M_n(\bb C)^{\st}$ associated to the generator-relation presentation of $M_n(\bb C)$ in terms of matrix units.  We thank Alec Fox for communicating the following fact to us:

\begin{prop}
Suppose that $A^\dagger$ is a computable presentation of $A$.  Then the universal presentation $(A\otimes M_n(\bb C))^{u(\dagger\otimes \st)}$ is computable.  Moreover, from $n$ and a code for $A^\dagger$, one can compute a code for $(A\otimes M_n(\bb C))^{u(\dagger\otimes \st)}$. 
\end{prop}

\begin{proof}
First, since $A^\dagger$ is computable, it is right-c.e., hence c.e. by Fact \ref{cefact}.  Conseequently, $(A\otimes M_n(\bb C))^{u(\dagger\otimes \st)}$ is c.e. by Lemma \ref{cetensor} and thus right-c.e. by Fact \ref{cefact} again.

It remains to show that $(A\otimes M_n(\bb C))^{u(\dagger\otimes \st)}$ is left-c.e.  To see this, we recall that, for any element $a=\sum_{i,j=1}^n a_{ij}\otimes e_{ij}\in A\otimes M_n(\bb C)$, we have $$\|a\|=\sup\left\|\sum_{i,j=1}^n x_ia_{ij}y_j^*\right\|,$$ where the supremum is taken over all pairs of $n$-tuples $(x_1,\ldots,x_n),(y_1,\ldots,y_n)\in A^n$ satisfying $\|\sum_{i=1}^n x_ix_i^*\|,\|\sum_{i=1}^n y_iy_i^*\|< 1$.  Consequently, if $a$ is a rational point of $(A\otimes M_n(\bb C))^{u(\dagger\otimes \st)}$, one can find computable lower bounds for $\|a\|$ by enumerating all such tuples consisting of rational points (which is possible since the norm of $A^\dagger$ is right-c.e.) and then computably approximating $\|\sum_{i,j=1}^n x_ia_{ij}y_j^*\|$ from below, which is possible since $A^\dagger$ is left-c.e. and the entries of $a$ are rational points of $A^\dagger$.  By Fact \ref{cefact} and the algorithm outlined here, this algorithm is uniform in $n$ and a code for $A^\dagger$.
\end{proof}





We now turn to presentations of inductive limits.  Suppose that $(A_m^{\dagger_m})_{m\in \bb N}$ is a sequence of universal presentations of \cstar-algebras $(A_m)_{m\in \bb N}$.  Suppose further that $\Phi_m:A_m\to A_{m+1}$ is a $*$-homomorphism.  If $a$ is the $n^{\text{th}}$ rational point of $A_m$ and $k\geq 1$, let $b(m,n,k)$ be the first rational point of $A_{m+1}^{\dagger_{m+1}}$ for which $\|\Phi_m(a)-b(m,n,k)\|<2^{-k}$.  We define the corresponding \textbf{inductive limit presentation} $\varinjlim (A_m^{\dagger_m},\Phi_m)$ of the inductive limit algebra $\varinjlim (A_m,\Phi_m)$ to be the universal presentation corresponding to the universal \cstar-algebra whose generators are the generators of the individual $A_m^{\dagger_m}$'s together with the relations of the various $A_m^{\dagger_m}$'s and relations stating $\|\Phi_m(a)-b(m,a,k)\|\leq 2^{-k}$ as above.

The following lemma has a routine proof:

\begin{lem}\label{indlimit}
In the notation of the previous paragraph, suppose that the sequence of maps $\Phi_m:A_m^{\dagger_m}\to A_{m+1}^{\dagger_{m+1}}$ is uniformly computable, that is, each map $\Phi_m:A_m^{\dagger_m}\to A_{m+1}^{\dagger_{m+1}}$ is computable and the function mapping $m$ to a code for $\Phi_m$ is computable.  Then:
\begin{enumerate}
    \item The map $b(m,n,k)$ is computable.
    \item If the presentations $A_m^\dagger$ are right-c.e. uniformly in $m$, then $\varinjlim (A_m^{\dagger_m},\Phi_m)$ is right-c.e.
    \item If the presentations $A_m^\dagger$ are computable uniformly in $m$ and each $\Phi_m$ is injective, then $\varinjlim (A_m^{\dagger_m},\Phi_m)$ is computable.
\end{enumerate}
\end{lem}

\subsection{Computable unitaries}

We begin this subsection with a result on ``almost unitaries'' that will be used in the following section.

Fix a \cstar-algebra $A$.  Recall that, for each $\epsilon\in (0,1]$, if $a\in A$ is an \textbf{$\epsilon$-almost unitary}, by which we mean that $\|a^*a-1\|,\|aa^*-1\|<\epsilon$, then $a$ is invertible and, letting $a=\omega(a)(a^*a)^{1/2}$ denote the polar decomposition of $a$, we have that $\omega(a)$ is a unitary and $\|a-\omega(a)\|<\epsilon$. 

Recall also the Taylor expansion 
$$x^{-1/2}=\sum_{k=0}^\infty \binom{k-1/2}{k}(1-x)^k, \quad |x|<1.$$

For each $n\geq 1$, set $s_n(x):=\sum_{k=0}^n \binom{k-1/2}{k}(1-x)^k$.  By Taylor's theorem, there is a computable function $\delta\mapsto N(\delta):\bb Q^{>0}\to \bb N$ such that $\|x^{-1/2}-s_{N(\delta)}(x)\|<\delta$ for all $x\in [1/2,3/2]$.  Consequently, if $a$ is an $\epsilon$-almost unitary with $\epsilon<1/2$ and $\|a\|\leq 1$, then since $\omega(a)=a(a^*a)^{-1/2}$, we have that $\|\omega(a)-as_{N(\delta)}(a^*a)\|<\delta$.  

The upshot of this discussion is the following:

\begin{lem}
There is a computable function such that:  for any presentation $A^\dagger$ of a \cstar-algebra $A$ for which $1$ is a computable point, upon input a code for a rational point $a$ of $A^\dagger$ with $\|a\|\leq 1$, a code for $1$, and $n\in \bb N$, if $\|a^*a-1\|,\|aa^*-1\|<1$, returns the code for a rational point  of $A^\dagger$, denoted $\omega_n(a)$, such that $\|\omega(a)-\omega_n(a)\|<2^{-n}$. 
\end{lem}

The remainder of this subsection contains results that will be used in the last section of this paper.

\begin{lem}\label{computablestone}
There is a computable map $e:\bb N^2\to \bb N$ such that, if $x$ is a code for a unitary element $u$ of $M_n(\bb C)$ that is a rational point of $M_n(\bb C)^{\st}$, then $e(n,x)$ is the code for a computable self-adjoint element $h$ of $M_n(\bb C)$ for which $\exp(ih)=u$.
\end{lem}

\begin{proof}
Recall that the Schur factorization of a matrix $a$ is a factorization of the form $a=zbz^*$, where $z$ is unitary and $b$ is upper triangular.  There are well-known algorithms for computing the Schur factorization of a matrix with rational complex entries.  If $a$ is itself unitary, then $b$ must be diagonal (being both upper triangular and unitary).  

Consequently, we can effectively find a diagonalization $u=zbz^*$ of the unitary $u$ coded by $x$.

Note also that one can compute a rational number $\theta\in [0,2\pi)$ such that $\exp(i\theta)$ is not an eigenvalue of $u$.  Let $\arg$ denote the branch of the argument function taking values in $[\theta,\theta+2\pi)$.  We can then set $h:=z\arg(b)z^*$, where $\arg(b)$ is the result of applying $\arg$ to each of the diagonal elements of $b$; note that $u=\exp(ih)$, $h$ is computable, and a code for $h$ can be found from a code for $b$. 
\end{proof}

\begin{remarks}
It is unclear to us if a version of the previous proof goes through if the assumption that $u$ is a rational unitary is replaced by the more general assumption that $u$ is a computable unitary.
\end{remarks}

In the following lemma, we equip $C[0,1]$ with its \textbf{standard presentation} $C[0,1]^{\st}$ consisting only of the identity function $\iota:=\id_{C[0,1]}$; we note that this presentation is computable.

Given any \cstar-algebra $A$, there is a unique isomorphism $$\phi_A:C[0,1]\otimes A\to C([0,1],A)$$ for which $(\phi_A(f\otimes a))(t)=f(t)a$ for all $f\in C[0,1]$ and $a\in A$.  If $A^\dagger$ is a presentation of $A$, then then we let $C([0,1],A)^\dagger$ be the presentation of $C([0,1],A)$ induced by the presentation $(C[0,1]\otimes A)^{\st\otimes \dagger}$ via $\phi_A$.

Suppose that $x,y\in \bb N$ are codes for computable unitaries $u$ and $v$ of $M_n(\bb C)$ respectively.  Let $h_u$ and $h_v$ be the self-adjoint elements coded by $e(n,x)$ and $e(n,y)$.  We let $u\leadsto v\in C([0,1],M_n(\bb C))$ be defined by $$(u\leadsto v)(t):=\exp(i(1-t)h_u)\exp(ith_v).$$  Note that $(u\leadsto v)(t)$ is a unitary element of $M_n(\bb C)$ for all  $t\in [0,1]$ satisfying $(u\leadsto v)(0)=u$ and $(u\leadsto v)(1)=v$.  

\begin{lem}\label{path}
There is a computable function $g:\bb N^3\to \bb N$ such that, if $x$ and $y$ are codes for rational unitaries $u$ and $v$ of $M_n(\bb C)$, then $g(n,x,y)$ is the code for $u\leadsto v$ with respect to the presentation $C([0,1],M_n(\bb C))^{\st}$. 
\end{lem}

\begin{proof}
Since $(u\leadsto v)(t)=(1\otimes u)\exp(ti(h_v-h_u))$, it suffices to show that, for any computable self-adjoint element $h$ of $M_n(\bb C)$, the function $t\mapsto \exp(ith)$ is a computable element of $C([0,1],M_n(\bb C))^{\st}$ uniformly in $n$ and a code for $h$.  However, since $\exp(ith)=\sum_{m=0}^\infty \frac{i^m}{m!}t^mh^m$ and the function $t\mapsto t^mh^m$ equals $\phi_{M_n(\bb C)}(\iota^m\otimes h^m)$, the desired conclusion follows from a Taylor estimate like that done earlier in this subsection.
\end{proof}


\section{Computable approximate intertwining}

In this section, we prove the computable analog of the approximate intertwining result of Elliott, following R\o rdam's exposition \cite{rordam}.

\begin{defn}
Two computable $*$-homomorphisms $\varphi,\psi:A^\dagger\to B^\#$ are \textbf{computably approximately unitarily equivalent} if there is a computable sequence $(v_n)_{n\in \bb N}$ from $B^\dagger$ consisting of unitaries such that there is an algorithm for which, given inputs a code for a rational point $a$ of $A^\dagger$ and $m\in \bb N$, returns $n\in \bb N$ such that $\|v_n\varphi(a)v_n^*-\psi(a)\|<2^{-m}$.
\end{defn}

\begin{remarks}\label{compapprunitary}

\

\begin{enumerate}
    \item There are ways of altering the previous definition that can lead to different notions of computably approximately unitarily equivalent morphisms.  For example, one might not require the sequence $(v_n)_{n\in \bb N}$ to be computable.  Alternatively, one might not require that $\varphi$ and $\psi$ be computable maps; however, with the rest of the definition unchanged, note that, if $\varphi$ and $\psi$ are computably approximately unitarily equivalent, then $\varphi$ is computable if and only if $\psi$ is computable.
  \item If $\varphi,\psi:A^\dagger\to B^\#$ are computable maps, $(v_n)_{n\in \bb N}$ is a computable sequence of unitaries witnessing that $\varphi$ and $\psi$ are approximately unitarily equivalent, and the presentation $B^\#$ is computable, then the last part of the definition is automatic. 
   \end{enumerate}
\end{remarks}

The following is the main result of this section and is the computable analog of \cite[Proposition 2.3.5]{rordam}:

\begin{thm}\label{maintheorem}
Let $A$ and $B$ be separable, unital \cstar-algebras and let $\varphi:A\to B$ be an injective $*$-homomorphism.  Equip $A$ and $B$  with computable presentations $A^\dagger$ and $B^\#$ respectively.  Suppose the following conditions hold:
\begin{enumerate}
\item $1$ is a computable point of $B^\#$.
\item $\varphi:A^\dagger\to B^\#$ is a computable map.
\item There is a sequence of unitaries $(w_n)_{n\in \bb N}$ from $B$ such that
$$\lim_{n\to \infty}\|w_n\varphi(a)-\varphi(a)w_n\|=0 \text{ and }\lim_{n\to \infty}d(w_n^*bw_n,\varphi(A))=0$$ for all $a\in A$ and $b\in B$.
\end{enumerate}
Then there is a computable isomorphism $\psi:A^\dagger\to B^\#$ that is computably approximately unitarily equivalent to $\varphi$.
\end{thm}

\begin{proof}
For ease of exposition, we assume that the $\varphi$-image of every rational point of $A^\dagger$ is actually a rational point of $B^\#$; the general case just involves an extra approximation step in what follows.  Let $(a_n)_{n\in \bb N}$ and $(b_n)_{n\in \bb N}$ be effective enumerations of the rational points of $A^\dagger$ and $B^\#$ respectively.

Begin by effectively choosing small rational numbers $\epsilon_1$ and $\eta_1$ (to be determined below) and searching for the first $\epsilon_1$-almost unitary rational point $\tilde{v}_1$ of $B^\#$ with $\|\tilde{v}_1\|\leq 1$ and first rational point $a_{1,1}$ of $A^\dagger$ such that
$$\max(\|\tilde{v}_1^*b_1\tilde{v_1}-\varphi(a_{1,1})\|,\|\tilde{v}_1\varphi(a_1)-\varphi(a_1)\tilde{v}_1\|)<\eta_1.$$  That such rational points exist follows from assumption (3) and that such computations can be done effectively follows from (1), (2) and the fact that the presentations are computable.  (We will omit such justifications in the sequel.)  Set $v_1:=\omega(\tilde{v}_1)$, so $\|v_1-\tilde{v}_1\|<\epsilon_1$.  It follows that $$\|v_1^*b_1v_1-\varphi(a_{1,1})\|< 2\epsilon_1\|b_1\|+\eta_1$$ and $$\|v_1\varphi(a_1)-\varphi(a_1)v_1\|<2\epsilon_1\|a_1\|+\eta_1.$$

Consequently, taking $\epsilon_1<1/(2^3\max(\|a_1\|,\|b_1\|))$ and $\eta_1<2^{-2}$ we have that $$\max(\|v_1^*b_1v_1-\varphi(a_{1,1})\|,\|v_1\varphi(a_1)-\varphi(a_1)v_1\|)<2^{-1}.$$

Next effectively choose small rational numbers $\epsilon_2$ and $\eta_2$ and a sufficiently large integer $k(2)$ (again, to be determined below) and search for the first $\epsilon_2$-almost unitary rational point $\tilde{v}_2$ of $B^\#$ with $\|\tilde{v}_2\|\leq 1$and first rational points $a_{1,2},a_{2,2}$ of $A^\dagger$ such that:
\begin{itemize}
    \item $\max_{j=1,2}\|\tilde{v}_2^*(\omega_{k(2)}(\tilde{v}_1)^*b_j\omega_{k(2)}(\tilde{v}_1))\tilde{v}_2-\varphi(a_{j,2})\|<\eta_2$.
    \item $\|\tilde{v}_2\varphi(a_{1,1})-\varphi(a_{1,1})\tilde{v}_2\|<\eta_2$.
    \item $\max_{j=1,2}\|\tilde{v}_2\varphi(a_j)-\varphi(a_j)\tilde{v}_2\|<\eta_2$.
\end{itemize}
Set $v_2:=\omega(\tilde{v}_2)$, so $\|v_2-\tilde{v}_2\|<\epsilon_2$.
For $j=1,2$, we then have:
\begin{itemize}
    \item $\|v_2(v_1^*b_jv_1)v_2)-\varphi(a_{j,2})\|<(2\epsilon_2+2^{-k(2)+1})\|b_j\|+\eta_2$.
    \item $\|v_2\varphi(a_{1,1})-\varphi(a_{1,1})v_2\|<2\epsilon_2\|a_{1,1}\|+\eta_2$.
    \item $\|v_2\varphi(a_j)-\varphi(a_j)v_2\|<2\epsilon_2\|a_{j}\|+\eta_2$.
\end{itemize}
One can computably choose $\epsilon_2$ and $\eta_2$ sufficiently small and $k(2)$ sufficiently large so that all of the bounds in the previous three bullets can be taken to be $2^{-2}$.

Inductively suppose that, for $j=1,\ldots,n-1$ and $k=1,\ldots,j$, one has found $\epsilon_j$-almost unitary rational points $\tilde{v}_j$ of $B^\#$ with $\|\tilde{v}_j\|\leq 1$ and rational points $a_{k,j}$ of $A^\dagger$ in the manner above.  Once again, effectively choose small rational numbers $\epsilon_n$ and $\eta_n$ and a sufficiently large integer $k(n)$ and search for the first $\epsilon_n$-almost unitary rational point $\tilde{v}_n$ of $B^\#$ with $\|\tilde{v}_n\|\leq 1$ and first rational points $a_{k,n}$ of $A^\dagger$ with $k=1,\ldots,n$ such that the following quantities are bounded by $\eta_n$, for all $1\leq j\leq n$ and $1\leq k\leq j$: 
\begin{itemize}
\item $\|\tilde{v}_n^*(\omega_{k(n)}(\tilde{v}_{n-1})^*\cdots \omega_{k(n)}(\tilde{v}_1^*)b_j\omega_{k(n)}(\tilde{v}_{1})\cdots \omega_{k(n)}(\tilde{v}_{n-1}))\tilde{v}_n-\varphi(a_{j,n})\|$
    \item $\|\tilde{v}_n\varphi(a_{k,j})-\varphi(a_{k,j})\tilde{v}_2\|$
    \item $\|\tilde{v}_n\varphi(a_j)-\varphi(a_j)\tilde{v}_n\|$.
    \end{itemize}

Set $v_n:=\omega(\tilde{v}_n)$ so $\|v_n-\tilde{v}_n\|<\epsilon_n$.  We then have the following inequalities for all $1\leq j\leq n$ and all $1\leq k\leq j$:
\begin{itemize}
    \item $\|v_n^*(v_{n-1}^*\cdots v_1^*b_jv_{1}\cdots v_{n-1})v_n-\varphi(a_{j,n})\|<(2\epsilon_n+2^{-k(n)+1})\|b_j\|+\eta_n$.
    \item $\|v_n\varphi(a_{k,j})-\varphi(a_{k,j})v_n\|<2\epsilon_n\|a_{k,j}\|+\eta_n$.
    \item $\|v_n\varphi(a_j)-\varphi(a_j)v_n\|<2\epsilon_n\|a_{j}\|+\eta_n$.
\end{itemize}
One can computably choose $\epsilon_n$ and $\eta_n$ sufficiently small and $k(n)$ sufficiently large so that all of the bounds in the previous three bullets can be taken to be $2^{-n}$.

For each $a\in A$ and $n\in \bb N$, set $w_n(a):=v_1v_2\cdots v_n\varphi(a)v_n^*\cdots v_2^*v_1^*$.  Note that $\|w_n(a)-w_{n+1}(a)\|=\|\varphi(a)-v_{n+1}\varphi(a)v_{n+1}^*\|<2^{-(n+1)}$ for all rational points $a$ of $A^\dagger$ and consequently for all $a\in A$.  Setting $\psi(a):=\lim_{n\to \infty}w_n(a)$, the proof of \cite[Proposition 2.3.5]{rordam} shows that $\psi:A\to B$ is an isomorphism.  We claim that $\psi$ is a computable map.  

Fix $j,m\in \bb N$.  Notice that for any $p\geq 1$, we have
$$\|w_{m+1}(a_j)-\omega_p(\tilde{v}_1)\cdots \omega_p(\tilde{v}_{m+1})\varphi(a_j)\omega_p(v_{m+1})^*\cdots \omega_p(v_1)^*\|<(m+1)2^{-p+1}\|a_j\|.$$
Consequently, taking $p$ such that $2^{p-m}>(m+1)\|a_j\|$, and noting that $\|w_{m+1}(a)-\psi(a)\|<2^{-(m+1)}$, we have that $\omega_{p}(\tilde{v}_1)\cdots \omega_p(\tilde{v}_n)\varphi(a_j)\omega_p(v_n)^*\cdots \omega_p(v_1)^*$ is a rational point of $B^\#$ within $2^{-m}$ of $\psi(a_j)$.  It follows that $\psi$ is a computable map.

It is clear from construction that the sequence $(v_1\cdots v_n)_{n\in \bb N}$ is computable; as shown in the proof of \cite[Proposition 2.3.5]{rordam}, the sequence also witnesses that $\varphi$ and $\psi$ are approximately unitarily equivalent.  By Remark \ref{compapprunitary}(2), the sequence witnesses that $\varphi$ and $\psi$ are computably approximately unitarily equivalent.
\end{proof}

\section{Computably strongly self-absorbing \cstar-algebras}

In this section, we apply the result of the previous section to the study of strongly self-absorbing \cstar-algebras.  We first recall that a \cstar-algebra $A$ is said to have \textbf{approximately inner half-flip} if the two inclusions $\id_A\otimes 1_A,1_A\otimes \id_A:A\hookrightarrow A\otimes A$ are approximately unitarily equivalent.  A proof of the following fact can be found in \cite[Theorem 7.2.2]{rordam}.  In what follows, $\u$ always denotes a nonprincipal ultrafilter on $\bb N$, $A^\u$ denotes the ultrapower $A$ with respect to $\u$, and $$A'\cap A^\u:=\{x\in A^\u \ : \ xy=yx \text{ for all }y\in A\}$$ denotes the relative commutant of $A$ inside of $A^\u$, where we identify $A$ with its image in $A^\u$ under the diagonal embedding.

\begin{fact}
Suppose that $A$ and $B$ are separable \cstar-algebras and $B$ has approximately inner half flip.  Further suppose that $B$ embeds into $A'\cap A^\u$.  Then the map $\id_A\otimes 1_B$ satisfies (3) in Theorem \ref{maintheorem}.
\end{fact}

Using Lemma \ref{computablembedding}, we have:

\begin{cor}
Suppose that $A^\dagger$ and $B^\#$ are presentations of \cstar-algebras $A$ and $B$.  Suppose that the following conditions hold:
\begin{enumerate}
\item $B$ has approximately inner half flip.
\item $B$ embeds into $A'\cap A^\u$.
\item $A^\dagger$ and $(A\otimes B)^{\dagger\otimes \#}$ are computable presentations.
\item $1_A$ and $1_B$ are computable points of $A^\dagger$ and $B^\#$ respectively.
\end{enumerate}
Then there is a computable isomorphism $\psi:A^\dagger \to (A\otimes B)^{\dagger\otimes \#}$ computably approximately unitarily equivalent to $\id_A\otimes 1_B$.
\end{cor}

As stated in the introduction, a \cstar-algebra $D$ is strongly self-absorbing if it is not isomorphic to $\bb C$ and there is an isomorphism $D\to D\otimes D$ approximately unitarily equivalent to $\id_D\otimes 1_D$.  Strongly self-absorbing \cstar-algebras have approximately inner half-flip \cite[Section 1]{TW}.
If $D$ is a strongly self-absorbing \cstar-algebra, then a \cstar-algebra $A$ is called \textbf{$D$-stable if $A\cong A\otimes D$}.  If $A$ is $D$-stable, then $D$ embeds in $A'\cap A^\u$ \cite[Theorem 7.2.2]{rordam}.  Consequently, we have:

\begin{cor}
Suppose that $A^\dagger$ and $D^\#$ are presentations of \cstar-algebras $A$ and $D$.  Suppose that the following conditions hold:
\begin{enumerate}
\item $D$ is strongly self-absorbing.
\item $A$ is $D$-stable.
\item $A^\dagger$ and $(A\otimes D)^{\dagger\otimes \#}$ are computable presentations.
\item $1_A$ and $1_D$ are computable points of $A^\dagger$ and $D^\#$ respectively.
\end{enumerate}
Then there is a computable isomorphism $\psi:A^\dagger \to (A\otimes D)^{\dagger\otimes \#}$ computably approximately unitarily equivalent to $\id_A\otimes 1_D$.
\end{cor}

Every strongly self-absoring algebra $D$ is $D$-stable.  Since strongly self-absorbing algebras are simple, any c.e. presentation of $D$ is automatically computable.  Moreover, if $D^\dagger$ is a c.e. presentation of $D$, then so is $(D\otimes D)^{u(\dagger\otimes \dagger)}$ by Lemma \ref{cetensor}, whence this presentation is also computable.

We say that a presentation $A^\dagger$ is \textbf{computably strongly self-absorbing} if there is a computable isomorphism $A^\dagger\to (A\otimes A)^{\dagger\otimes \dagger}$ that is computably approximately unitarily equivalent to $\id_A\otimes 1_A$.  Note that, since our definition of computably approximately unitarily equivalent morphisms requires that the morphisms be computable, the definition of computably strongly self-absorbing presentation presupposes that the map $\id_A\otimes 1_A:A^\dagger\to (A\otimes A)^\dagger$ is computable, which happens, for example, when $A^\dagger$ is bounded and $1$ is a rational point of $A^\dagger$ (see Lemma \ref{computablembedding}).  The previous discussion and Lemma \ref{doesntmatter} imply:

\begin{cor}
Suppose that $D$ is a strongly self-absorbing \cstar-algebra and $D^\dagger$ is a c.e. (and thus computable) presentation of $D$.  Then $D^\dagger$ is computably strongly self-absorbing.
\end{cor}


The standard presentations of the Cuntz algebras $\cal O_2$ and $\cal O_\infty$ are clearly c.e.  Consequently, we have:

\begin{cor}
$\O_2^{\st}$ and $\cal O_\infty^{\st}$ are computably strongly self-absorbing.
\end{cor}

As mentioned in \cite[Examples 1.14]{TW}, a UHF algebra $M_{\frak n}(\bb C)$ is strongly self-absorbing if and only if the supernatural number $\frak n$ is of \textbf{infinite type}, that is, if all of its nonzero exponents are infinite (and at least one exponent is nonzero), in which case $M_{\frak n}(\bb C)\otimes \cal O_\infty$ is also strongly self-absorbing.  By the \textbf{support} of a supernatural number $\frak n$, we mean the set of primes which appear in $\frak n$ with nonzero exponents.  By \cite{EGMM}, a UHF algebra $M_{\frak n}(\bb C)$ of infinite type has a computable presentation if and only if its support is c.e., in which case it admits a computable ``standard presentation'' $M_{\frak n}(\bb C)^{\st}$. Together with Lemma \ref{cetensor}, we have:

\begin{cor}
Suppose that $\frak n$ is a supernatural number of infinite type with c.e. support.  Then $M_{\frak n}(\bb C)^{\st}$ and $(M_{\frak n}(\bb C)\otimes \cal O_\infty)^{\st\otimes \st}$ are computably strongly self-absorbing.
\end{cor}

Any strongly self-absorbing \cstar-algebra satisfying the UCT must be isomorphic to $M_{\frak n}(\bb C)$, $\cal O_\infty$, $M_{\frak n}(\bb C)\otimes \cal O_\infty$, $\cal Z$, or $\cal O_2$, where $\frak n$ is of infinite type and $\cal Z$ is the \textbf{Jiang-Su algebra}.  In the next section, we show that $\cal Z$ has a ``standard presentation'' that is computable, whence it is also computably strongly self-absorbing.  Consequently, with the possible exception of the UCT strongly self-absorbing algebras involving UHF algebras without computable presentations, all of the known UCT strongly self-absorbing \cstar-algebras are computably strongly self-absorbing. 

If $D$ is any strongly self-absorbing \cstar-algebra, then  $D\otimes \cal Z\cong D$ and $D\otimes \cal O_2\cong \cal O_2$.  Consequently, we have:
\begin{cor}
Suppose that $D$ is a strongly self-absorbing \cstar-algebra that admits a c.e. presentation $D^\dagger$.  Then:
\begin{enumerate}
    \item There is a computable isomorphism $D^{\dagger}\to (D\otimes \cal Z)^{\dagger\otimes \st}$ computably approximately unitarily equivalent to $\id_{D}\otimes 1_{\cal Z}$.
    \item There is a computable isomorphism $\cal O_2^{\st}\to (D\otimes \cal O_2)^{\dagger\otimes\st}$ computably approximately unitarily equivalent to $\id_{\cal O_2}\otimes 1_{D}$. 
\end{enumerate}
\end{cor} 


\section{A computable presentation of the Jiang-Su algebra}

In this section, we show how the original construction of the Jiang-Su algebra $\cal Z$ given in \cite{jiangsu} yields a ``standard'' computable presentation $\cal Z^{\st}$ of $\cal Z$, whence, by the results of the previous section, shows that $\cal Z$ is also computably strongly self-absorbing.

Given integers $p,q\geq 2$, we let $\cal Z_{p,q}$ denote the functions $f\in C([0,1],M_p(\bb C)\otimes M_q(\bb C))$ such that $f(0)\in M_p(\bb C)\otimes 1_{M_q(\bb C)}$ and $f(1)\in 1_{M_q(\bb C)}\otimes M_q(\bb C)$; $\cal Z_{p,q}$ is a \cstar-subalgebra of $C([0,1],M_p(\bb C)\otimes M_q(\bb C))$ called a \textbf{dimension drop algebra}; when $p$ and $q$ are relatively prime, $\cal Z_{p,q}$ is called a \textbf{prime} dimension drop algebra.

By \cite[Proposition 7.3]{jiangsu}, $\cal Z_{p,q}$ admits a generator-relations presentation consisting of finitely many generators and relations; as usual, we let $\cal Z_{p,q}^{\st}$ denote the associated standard presentation of $\cal Z_{p,q}$.  The exact details of the presentation are not relevant for us; the only thing we will need is the following:

\begin{lem}\label{generatorsarecomputable}
Viewing $\cal Z_{p,q}$ as a subalgebra of $C[0,1]\otimes M_p(\bb C)\otimes M_q(\bb C)$, the generators of $\cal Z_{p,q}^{\st}$ are computable points of $(C[0,1]\otimes M_p(\bb C)\otimes M_q(\bb C))^{u(\st\otimes \st\otimes \st)}$, uniformly in $p$ and $q$.
\end{lem}

\begin{proof}
The generators of $\cal Z_{p,q}$ are $a_i:=(1-\iota)^{1/2}\otimes e^{(p)}_{1i}\otimes 1_q$, $i=1,\ldots,p$, and $b_j:=\iota^{1/2}\otimes 1_p\otimes e^{(q)}_{1j}$, $j=1,\ldots,q$.  The result follows from the fact that $\iota^{1/2}$ and $(1-\iota)^{1/2}$ are computable points of $C[0,1]^{\st}$.
\end{proof}

The Jiang-Su algebra $\cal Z$ is a particular inductive limit $\varinjlim (\cal Z_{p_m,q_m},\Phi_m)$ of  prime dimension drop algebras, which we now describe.  Begin by setting $p_0=2$ and $q_0=3$ and setting $A_0:=\cal Z_{p_0,q_0}$.  Supposing that the prime dimension drop algebra $A_m=:\cal Z_{p_m,q_m}$ has been constructed, we now construct a new dimension drop algebra $A_{m+1}:=\cal Z_{p_{m+1},q_{m+1}}$ and an injective $*$-homomorphism $\Phi_m:A_m\to A_{m+1}$.  Let $k_m$ and $l_m$ denote the first two prime numbers larger than $2p_mq_m$.\footnote{N.B. Our notation is slightly different than that in \cite{jiangsu}.}  We set $p_{m+1}:=k_mp_m$ and $q_{m+1}:=l_mq_m$.  Note that $p_{m+1}$ and $q_{m+1}$ are relatively prime, so $A_{m+1}$ is a prime dimension drop algebra.  We now construct the morphism $\Phi_m$.

Set $r_m$ to be the remainder of $k_ml_m$ modulo $q_{m+1}$ and set $s_m$ to be the remainder of $k_ml_m$ modulo $p_{m+1}$.  For $i=1,\ldots,k_ml_m$, we define functions $\xi_i\in C[0,1]$ as follows:  for $i=1,\ldots,r_m$, set $\xi_i(t)=t/2$; for $i=r_m+1,\ldots,k-s_m$, set $\xi_i(t)=1/2$; for $i=k-s_m+1,\ldots,k_ml_m$, set $\xi_i(t)=(t+1)/2$.

As noted in the proof of \cite[Proposition 2.5]{jiangsu}, both $r_mq_m$ and $k_ml_m-r_m$ are divisible by $q_{m+1}$, say $r_mq_m=\alpha_m q_{m+1}$ and $k_ml_m-r_m=\beta_m q_{m+1}$.  Let $u_m\in M_{p_{m+1}q_{m+1}}(\bb C)$ be the unitary matrix taking $$\diag(f(\xi_1(0)),\ldots,f(\xi_{k_ml_m}(0)))$$ to the block diagonal matrices with blocks $$\diag(f(0),\ldots,f(0),f(1/2),\ldots,f(1/2)),$$ where $f(0)$ is repeated $\alpha_m$ times and $f(1/2)$ is repeated $\beta_m$ times.  Moreover, $u_m$ is a rational point of $M_{p_{m+1}q_{m+1}}(\bb C)^{\st}$ for which the map $$f\mapsto u_m^*\diag(f(\xi_1(0)),\ldots,f(\xi_{k_ml_m}(0))u_m$$ is a $*$-homomorphism $A_m\to M_{p_{m+1}}\otimes 1_{M_{q_{m+1}}}$.   Similarly, there is a computable unitary $v_m$ such that $f\mapsto v_m^*\diag(f(\xi_1(1)),\ldots,f_{k_ml_m}(1))v_m$ is a $*$-homomorphism $A_m\to 1_{M_{p_{m+1}}}(\bb C)\otimes M_{q_{m+1}}(\bb C)$.

Then the map $\Phi_m(f)=(u_m\leadsto v_m)^*\diag (f\circ \xi_1,\ldots,f\circ \xi_{k_ml_m})(u_m\leadsto v_m)$ is the desired injective $*$-homomorphism $\Phi_m:A_m\to A_{m+1}$.  

\begin{prop}
The map $\Phi_m:\cal Z_{p_m,q_m}^{\st}\to \cal Z_{p_{m+1},q_{m+1}}^{\st}$ is computable, uniformly in $m$.
\end{prop}

\begin{proof}
By Lemma \ref{generatorsarecomputable}, it suffices to show that $$\Phi_m:\cal Z_{p_m,q_m}^{\st}\to (C[0,1]\otimes M_{p_{m+1}}(\bb C)\otimes M_{q_{m+1}}(\bb C))^{u(\st\otimes \st\otimes \st)}$$ is computable uniformly in $m$.  This follows from Lemmas \ref{path} and \ref{generatorsarecomputable}, the fact that the map $f\mapsto \diag (f\circ \xi_1,\ldots,f\circ \xi_{k_ml_m})$ is a computable map $$(C[0,1]\otimes M_{p_{m}}(\bb C)\otimes M_{q_{m}}(\bb C))^{u(\st\otimes \st\otimes \st)}\to (C[0,1]\otimes M_{p_{m+1}}(\bb C)\otimes M_{q_{m+1}}(\bb C))^{u(\st\otimes \st\otimes \st)},$$ and that the unitaries $u_m$ and $v_m$ are rational points of $M_{p_{m+1}}(\bb C)^{\st}$ and $M_{q_{m+1}}(\bb C)^{\st}$ respectively, uniformly in $m$.
\end{proof}

Set $\cal Z^{\st}:=\varinjlim (\cal Z_{p_m,q_m}^{\st},\Phi_m)$, a so-called ``standard presentation'' of $\cal Z$.  Combining the above discussion with Fact \ref{alec} and Lemma \ref{indlimit}, we arrive at:

\begin{thm}
The presentation $\cal Z^{\st}$ is computable and thus computably strongly self-absorbing.
\end{thm}

\end{document}